\documentclass[12pt,a4paper,oneside,reqno]{article}

\usepackage[utf8]{inputenc}
\usepackage[english]{babel}
\usepackage[left=1in,right=1in,top=1in,bottom=1in]{geometry}
\usepackage{amsfonts,amssymb,amsmath,amsthm}
\usepackage[symbol]{footmisc}
\usepackage{longtable,array}
\newcolumntype{L}[1]{>{\raggedright\arraybackslash$}p{#1}<{$}}
\usepackage{enumerate}
\usepackage{cite}
\usepackage{xcolor}

\usepackage{hyperref}
\usepackage{fancyhdr}
\theoremstyle{plain}
\newtheorem{theorem}{Theorem}
\newtheorem{lemma}[theorem]{Lemma}
\newtheorem{proposition}[theorem]{Proposition}
\newtheorem{corollary}[theorem]{Corollary}
\theoremstyle{definition}
\newtheorem{remark}[theorem]{Remark}
\newtheorem{definition}[theorem]{Definition}

\makeatletter
\@removefromreset{theorem}{section}\@removefromreset{equation}{section}

\makeatother
\usepackage{adjustbox}
\usepackage{tikz}\usetikzlibrary{arrows}

\IfFileExists{fouriernc.sty}{\usepackage{fouriernc}}{}

\newcommand{\doi}[1]{\href{https://doi.org/#1}{doi:#1}}
\newcommand{\arxiv}[1]{\href{https://arxiv.org/abs/#1}{arXiv:#1}}
\newcommand{\arxivsearch}[1]{\href{https://arxiv.org/search/?query=#1}{arXiv}}

\begin{document}

\noindent\fbox{\parbox{\dimexpr\textwidth-2\fboxsep-2\fboxrule\relax}{\small
\textbf{Corrected version (verified by Claude).} This is a corrected version of the paper B. L. M. Ferreira, I. Kaygorodov, V. Lopatkin,
\emph{$\frac{1}{2}$-derivations of Lie algebras and transposed Poisson algebras},
Revista de la Real Academia de Ciencias Exactas, F\'isicas y Naturales. Serie A. Matem\'aticas, 115 (2021), no.~3, 142,
\doi{10.1007/s13398-021-01088-2}.}}

 \medskip

\begin{center}
{\Large {\bf $\frac{1}{2}$-derivations of Lie algebras and\\ transposed Poisson algebras}}

\medskip

{\bf Bruno Leonardo Macedo Ferreira, Ivan Kaygorodov and Viktor Lopatkin}
\end{center}

\noindent{\small\textbf{Abstract.} A relation between $\frac{1}{2}$-derivations of Lie algebras and 
transposed Poisson algebras has been established.
Some non-trivial transposed Poisson algebras with a certain Lie algebra 
(Witt algebra, the algebra $\mathcal{W}(a,-1)$, the thin Lie algebra and a solvable Lie algebra with abelian nilpotent radical) have been done.
In particular, we have developed an example of the transposed Poisson algebra with associative and Lie parts isomorphic to the
Laurent polynomials and the Witt algebra.
On the other side, it has been proved that there are no non-trivial transposed Poisson algebras with a Lie algebra part isomorphic to 
a semisimple finite-dimensional algebra, a simple finite-dimensional superalgebra, the Virasoro algebra,
$N=1$ and $N=2$ superconformal algebras, or a semisimple finite-dimensional $n$-Lie algebra.}

\medskip

\noindent {\bf Keywords:}
{\it $\delta$-derivation, Lie algebra, transposed Poisson algebra, Witt algebra.}

\noindent {\bf MSC2020:} 17B63 (primary); 17A30, 17B40, 17B65, 17B68 (secondary).

\tableofcontents

\section{Introduction}
Poisson algebras arose from the study of Poisson geometry in the 1970s and have appeared in an extremely wide range of areas in mathematics and physics, such as Poisson manifolds, algebraic geometry, operads, quantization theory, quantum groups and classical and quantum mechanics. The study of Poisson algebras also led to other algebraic structures, such as 
noncommutative Poisson algebras \cite{kubo}, 
generic Poisson algebras \cite{ksu18},
algebras of Jordan brackets and generalized Poisson algebras \cite{k17,ck10},
Gerstenhaber algebras \cite{kos},
Novikov-Poisson algebras \cite{xu},
Malcev-Poisson-Jordan algebras \cite{saidmal}, 
double Poisson algebras \cite{vd},
$n$-ary Poisson algebras \cite{ck16}, etc.
The study of all possible Poisson structures with a certain Lie or associative part is an important problem in the theory of Poisson algebras \cite{jawo,said,said2}.
Recently, a dual notion of the Poisson algebra (transposed Poisson algebra) by exchanging the roles of the two binary operations in the Leibniz rule defining the Poisson algebra has been introduced in the paper of Bai, Bai, Guo and Wu \cite{bai20}. 
They have shown that the transposed Poisson algebra defined this way not only shares common properties of the Poisson algebra, including the closure undertaking tensor products and the Koszul self-duality as an operad but also admits a rich class of identities. More significantly, a transposed Poisson algebra naturally arises from a Novikov-Poisson algebra by taking the commutator Lie algebra of the Novikov algebra. Consequently, the classic construction of a Poisson algebra from a commutative associative algebra with a pair of commuting derivations has a similar construction of a transposed Poisson algebra when there is one derivation. More broadly, the transposed Poisson algebra also captures the algebraic structures when the commutator is taken in pre-Lie Poisson algebras and two other Poisson type algebras. 

The study of $\delta$-derivations of Lie algebras was initiated by Filippov in 1998 \cite{fil1,fil2,fil}. 
The space of $\delta$-derivations includes usual derivations, antiderivations and elements from the centroid.
During last 20 years, $\delta$-derivations of prime Lie algebras \cite{fil1, fil2},
$\delta$-derivations of simple Lie and Jordan superalgebras \cite{kay09, kay10, kay12mz,zus10},
$\delta$-derivations of semisimple Filippov ($n$-Lie) algebras \cite{kay12izv, kay14mz} have been studied.
$\delta$-Derivations is a particular case of generalized derivations (see, \cite{dgmrs20,ll,Beites}).

In the present paper, we have found a way how all transposed Poisson algebra structures with a certain Lie algebra can be described.
Our main tool is a description of the space of $\frac{1}{2}$-derivations of a certain Lie algebra
and establishing a connection between the space of $\frac{1}{2}$-derivations of the Lie part of a transposed Poisson algebra and the space of right multiplications of the associative part of this transposed Poisson algebra. Namely, every right {\it associative} multiplication is a {\it Lie} $\frac{1}{2}$-derivation.
Using the known description of $\delta$-derivations of semisimple finite-dimensional Lie algebras, 
we have found that there are no transposed Poisson algebras with a semisimple finite-dimensional Lie part.
In the case of simple infinite-dimensional algebras, we have a different situation.
It has been proved that the Witt algebra admits many nontrivial structures of transposed Poisson algebras.
Later we studied structures of transposed Poisson algebras defined on one of the most interesting generalizations of the Witt algebra. 
Namely, we have considered $\frac{1}{2}$-derivations of the algebra ${\mathcal W}(a,b)$ and have proved that the algebra 
${\mathcal W}(a,b)$ does not admit structures of transposed Poisson algebras if and only if $b\neq -1.$
All transposed Poisson algebra structures defined on ${\mathcal W}(a,-1)$ have been described. 
In the next section another example of an algebra well related to the Witt algebra has been considered.
We have proved that there are no transposed Poisson algebras defined on the Virasoro algebra.
As some corollaries we have proved that there are no transposed Poisson algebras structures defined on $N=1$ and $N=2$ superconformal algebras. The rest of the paper is dedicated to a classification of $\frac{1}{2}$-derivations and constructions of transposed Poisson algebras defined on the 
thin Lie algebra and the solvable Lie algebra with abelian nilpotent radical of codimension $1.$

\section{$\frac{1}{2}$-derivations of Lie algebras 
and transposed Poisson algebras}

All algebras and vector spaces we consider over the complex field but many results can be proven over other fields without modifications of proofs.

\subsection{Binary case}
Let us give some important definitions for our consideration.
The definition of the transposed Poisson algebra was given in a paper of Bai, Bai, Guo and Wu \cite{bai20}.
The definition of $\frac{1}{2}$-derivations as a particular case of $\delta$-derivations was given in a paper of Filippov \cite{fil1}.

\begin{definition}\label{tpa}
Let ${\mathfrak L}$ be a vector space equipped with two nonzero bilinear operations $\cdot$ and $[\;,\;].$
The triple $({\mathfrak L},\cdot,[\;,\;])$ is called a transposed Poisson algebra if $({\mathfrak L},\cdot)$ is a commutative associative algebra and
$({\mathfrak L},[\;,\;])$ is a Lie algebra that satisfies the following compatibility condition
\begin{equation}\label{link1}
2z\cdot [x,y]=[z\cdot x,y]+[x,z\cdot y].\end{equation}
\end{definition}

The dual Leibniz rule (\ref{link1}) gives the following trivial proposition.

\begin{proposition}
Let $({\mathfrak L}, \cdot, [\;,\;])$ be a transposed Poisson algebra. Then
\begin{enumerate}
 \item if $({\mathfrak L}, [\;,\;])$ is non-perfect $($i.e. $[{\mathfrak L}, {\mathfrak L}] \neq {\mathfrak L})$ then $({\mathfrak L}, \cdot)$ is non-simple;
 \item if $({\mathfrak L}, \cdot)$ is simple then $({\mathfrak L}, [ \ , \ ])$ is perfect $($i.e. $[{\mathfrak L}, {\mathfrak L}] = {\mathfrak L}).$
 
\end{enumerate} 
\end{proposition}

\begin{proof}
Note that the set $ [{\mathfrak L}, {\mathfrak L}] \neq 0$ gives an ideal of the associative algebra $({\mathfrak L}, \cdot)$
and hence we have the statement of the Proposition.
\end{proof}

\begin{definition}
Let $({\bf A}, \cdot)$ be an arbitrary associative algebra, 
and let $p$ and $q$ be two fixed elements of ${\bf A}.$ Then a
new algebra is derived from ${\bf A}$ by using the same vector space structure of ${\bf A}$
but defining a new multiplication
$$x * y = x\cdot p\cdot y - y\cdot q\cdot x$$
for $x, y \in {\bf A}.$ The resulting algebra is denoted by ${\bf A}(p, q)$ and is called the
$(p, q)$-mutation of the algebra ${\bf A}$
(for more information about mutations see, for example, \cite{elduque91}). 
In the commutative case, all $(p, q)$-mutations can be reduced to the case $q=0.$
\end{definition}


\begin{lemma}
Let $({\mathfrak L}, \cdot, [\;,\;])$ be a transposed Poisson algebra.
Then every mutation $({\mathfrak L}, \cdot_p)$ 
of $({\mathfrak L}, \cdot)$ gives a transposed Poisson algebra 
$({\mathfrak L}, \cdot_p, [\;,\;])$ with the same Lie multiplication.
\end{lemma}

\begin{proof}
Every mutation of an associative commutative algebra gives an associative commutative algebra.
  By the dual Leibniz identity,  
\begin{longtable}{lll}
$2z \cdot_p [x,y]$&$=$&$
2(z \cdot p) \cdot [x,y]=
[(z\cdot p)\cdot x,y]+[x,(z \cdot p) \cdot y]$\\
&$=$&$[z \cdot_p x,y]+[x,z\cdot_p y], $
\end{longtable}
 and the statement follows.  
\end{proof}

\begin{definition}\label{12der}
Let $({\mathfrak L}, [\;,\;])$ be an algebra with multiplication $[\;,\;]$ and $\varphi$ be a linear map.
Then $\varphi$ is a $\frac{1}{2}$-derivation if it satisfies
\begin{equation}
\varphi[x,y]= \frac{1}{2} \left([\varphi(x),y]+ [x, \varphi(y)] \right).\end{equation}
\end{definition}

Summarizing Definitions \ref{tpa} and \ref{12der} we have the following key lemma.
\begin{lemma}\label{glavlem}

Let $({\mathfrak L},\cdot,[\;,\;])$ be a transposed Poisson algebra 
and $z$ an arbitrary element from ${\mathfrak L}.$
Then the right multiplication $R_z$ in the associative commutative algebra $({\mathfrak L},\cdot)$ gives a $\frac{1}{2}$-derivation of the Lie algebra $({\mathfrak L}, [\;,\;]).$
\end{lemma}

The main example of $\frac{1}{2}$-derivations is the multiplication by an element from the ground field.
Let us call such $\frac{1}{2}$-derivations as trivial
$\frac{1}{2}$-derivations.
As it  follows from the following theorem we are not interested in trivial $\frac{1}{2}$-derivations.

\begin{theorem}\label{princth}
Let ${\mathfrak L}$ be a Lie algebra without non-trivial $\frac{1}{2}$-derivations.
Then every transposed Poisson algebra structure defined on ${\mathfrak L}$ is trivial.
\end{theorem}

\begin{proof}
Let $({\mathfrak L},\cdot,[\;,\;])$ be a transposed Poisson algebra.
Then $\varphi_y (x)= x \cdot y = \varphi_x (y),$ where $\varphi_x$ and $\varphi_y$ are some $\frac{1}{2}$-derivations of the Lie algebra $({\mathfrak L}, [\;,\;]).$
Hence, there are $\kappa_x, \kappa_y \in \mathbb{C},$ such that 
$\varphi_x(z)= \kappa_x z$ and $\varphi_y(z)= \kappa_y z$ for all $z \in {\mathfrak L}.$
If $({\mathfrak L},\cdot,[\;,\;])$ is non-trivial then $dim \ {\mathfrak L} >1$ and we can take $x$ and $y$ as some linear independent elements. 
It follows that $\kappa_x=\kappa_y=0$ and $x\cdot y=0$ for all elements $x,y \in {\mathfrak L}.$

\end{proof}

The description of $\frac{1}{2}$-derivations of simple finite-dimensional Lie algebras is given in \cite{fil2}.
There are no non-trivial $\frac{1}{2}$-derivations of simple finite-dimensional Lie algebras. Since every complex semisimple finite-dimensional Lie algebra is a direct sum of simple algebras then every $\frac{1}{2}$-derivation of this semisimple algebra will be invariant on all simple algebras from this direct sum
(see, for example, \cite[proof of Theorem 6]{kay14mz}).
Hence, if ${\mathfrak L}={\mathfrak L}_1\oplus\cdots\oplus{\mathfrak L}_s$ is the decomposition of a semisimple Lie algebra ${\mathfrak L}$ into simple ideals and $\pi_i$ is the projection of ${\mathfrak L}$ onto ${\mathfrak L}_i$, then every $\frac{1}{2}$-derivation of ${\mathfrak L}$ has the form $\sum_{i=1}^s\kappa_i\pi_i$ with $\kappa_i\in\mathbb C$. In particular, for $s>1$ the algebra ${\mathfrak L}$ has non-trivial $\frac{1}{2}$-derivations, and Theorem~\ref{princth} cannot be applied directly. Nevertheless, the idea of its proof gives the following corollary.

\begin{corollary}\label{corss}
There are no non-trivial transposed Poisson algebra structures defined on a complex semisimple finite-dimensional Lie algebra.
\end{corollary}

\begin{proof}
Let $({\mathfrak L},\cdot,[\;,\;])$ be a transposed Poisson algebra with a semisimple Lie part ${\mathfrak L}={\mathfrak L}_1\oplus\cdots\oplus{\mathfrak L}_s$.
By Lemma~\ref{glavlem} and the description above, $R_z=\sum_{i=1}^s\kappa_i(z)\pi_i$ for every $z\in{\mathfrak L}$, where $\kappa_1,\ldots,\kappa_s$ are linear functionals on ${\mathfrak L}$.
If $x,y\in{\mathfrak L}_i$ are linearly independent (recall that $\dim{\mathfrak L}_i\geq3$), then $\kappa_i(y)x=x\cdot y=y\cdot x=\kappa_i(x)y$, hence $\kappa_i(x)=0$ for all $x\in{\mathfrak L}_i$.
If $x\in{\mathfrak L}_i$ and $y\in{\mathfrak L}_j$ with $i\neq j$, then $\kappa_i(y)x=x\cdot y=y\cdot x=\kappa_j(x)y$, hence $\kappa_i(y)=0$.
Thus all $\kappa_i$ are zero and $x\cdot y=0$ for all $x,y\in{\mathfrak L}$.
\end{proof}

The next corollary is about a non-semisimple Lie algebra without 
non-trivial transposed Poisson algebra structures.
The Schrödinger algebra ${\mathfrak s}={\mathfrak{sl}}_2 \ltimes {\mathfrak H}$ is a semidirect product of the $3$-dimensional Lie algebra ${\mathfrak {sl}}_2$ and the $3$-dimensional Heisenberg (nilpotent Lie) algebra ${\mathfrak H}$ (see, for example, \cite{yu18}). It is easy to see that all $\frac{1}{2}$-derivations of ${\mathfrak s}$ are trivial and we have the following statement.

\begin{corollary}\label{corlie}
There are no non-trivial transposed Poisson algebra structures defined on the Schrödinger algebra.
\end{corollary}

\subsection{Supercase}
Superization of the definition of the transposed Poisson algebra should be given by the usual way
(for many types of superstructures see, for example, the brilliant paper of Kac \cite{kac}
or the brilliant survey of Leites \cite{Leites} and references therein).

\begin{definition}
Let ${\mathfrak L}={\mathfrak L}_0 \oplus {\mathfrak L}_1$ be a $\mathbb{Z}_2$-graded vector space 
equipped with two nonzero bilinear super-operations $\cdot$ and $[\;,\;].$
The triple $({\mathfrak L},\cdot,[\;,\;])$ is called a transposed Poisson superalgebra if 
$({\mathfrak L},\cdot)$ is a supercommutative associative superalgebra and
$({\mathfrak L},[\;,\;])$ is a Lie superalgebra that satisfies the following compatibility condition
\begin{equation}
2z\cdot [x,y]=[z\cdot x,y]+ (-1)^{|x||z|}[x,z\cdot y], \ x,y,z \in {\mathfrak L}_0 \cup {\mathfrak L}_1.\end{equation}
\end{definition}

As a first result in the study of transposed Poisson superalgebras we have the following statement.
It is a super-analog of Theorem \ref{princth} with a similar proof of it.

\begin{lemma}\label{superteo}
Let ${\mathfrak L}$ be a Lie superalgebra without non-trivial $\frac{1}{2}$-superderivations.
Then every transposed Poisson superalgebra structure defined on ${\mathfrak L}$ is trivial.
\end{lemma}

Thanks to \cite{kay09,kay10} we have the description of all 
$\frac{1}{2}$-superderivations of complex simple finite-dimensional Lie superalgebras.
All these $\frac{1}{2}$-superderivations are trivial.
Hence summarizing results of \cite{kay09,kay10} and Lemma \ref{superteo}
we have the following corollary.

\begin{corollary}\label{corsuper}
There are no non-trivial transposed Poisson superalgebra structures defined on complex simple finite-dimensional Lie superalgebras.
\end{corollary}

\subsection{$n$-ary case}
The idea of the $n$-ary generalization of transposed Poisson algebra is based on the idea of $n$-Nambu-Poisson algebra, introduced by Takhtajan \cite{takh},
and it also was introduced in \cite{bai20}.
The definition of $\frac{1}{n}$-derivations of $n$-ary algebras as a particular case of $\delta$-derivations of $n$-ary algebras was given in a paper of Kaygorodov \cite{kay12izv}.

\begin{definition}
Let $n\geq 2$ be an integer. 
A transposed Poisson $n$-Lie algebra is a triple $({\mathfrak L},\cdot, [\;, \ldots, \;])$ where 
$({\mathfrak L},\cdot)$ is a commutative associative algebra and $({\mathfrak L},[\;, \ldots, \;])$ is an $n$-Lie (Filippov) algebra satisfying the following condition.
$$ n \ w \cdot [x_1, \ldots, x_n]=\sum_{i=1}^n [x_1, \ldots, w \cdot x_i, \ldots, x_n].$$
\end{definition}

\begin{remark}
The super-analog of transposed Poisson $n$-Lie algebras should be defined in the usual way
(for more information about $n$-Lie superalgebras; see, for example, \cite{caka10,Leites}).
\end{remark}

\begin{definition}
Let $({\mathfrak L},[\;, \ldots, \;])$ be an $n$-ary algebra with the multiplication $[\;, \ldots, \;]$ and $\varphi$ be a linear map.
Then $\varphi$ is a $\frac{1}{n}$-derivation if it  satisfies
\begin{equation}\varphi[x_1,\ldots, x_n]= \frac{1}{n}\sum_{i=1}^n [x_1, \ldots, \varphi( x_i), \ldots, x_n],\;\;\forall x_1,\ldots, x_n \in L.\end{equation}
The main example of $\frac{1}{n}$-derivations is the multiplication on an element from the basic field.
Let us call such $\frac{1}{n}$-derivations as trivial
$\frac{1}{n}$-derivations. 
\end{definition}

Now we are ready to talk about $n$-ary analogs of Corollary \ref{corsuper}.

\begin{theorem}
There are no non-trivial transposed Poisson $n$-Lie algebra (or superalgebra) structures defined on 
complex simple finite-dimensional $n$-Lie algebras (or superalgebras).
\end{theorem}
\begin{proof}
Thanks to \cite{caka10}, every complex simple finite-dimensional $n$-Lie $(n>2)$ algebra (or superalgebra) 
is isomorphic to the $(n+1)$-dimensional $n$-ary algebra $A_{n+1}.$
All $\frac{1}{n}$-derivations of $A_{n+1}$ are trivial \cite{kay14mz}.
Then following the proof of Theorem \ref{princth} we have that the transposed Poisson $n$-Lie algebra is trivial.
\end{proof}

Since every complex semisimple finite-dimensional $n$-Lie algebra is a direct sum of simple algebras \cite{ling},
and as follows, every $\frac{1}{n}$-derivation of this semisimple algebra will be invariant on all simple algebras from this direct sum \cite{kay14mz}, we then have the following corollary.

\begin{corollary}
There are no non-trivial transposed Poisson $n$-Lie algebra structures defined on a complex semisimple finite-dimensional $n$-Lie algebra.
\end{corollary}

\begin{proof}
Let ${\mathfrak L}={\mathfrak L}_1\oplus\cdots\oplus{\mathfrak L}_s$ be the decomposition into simple ideals. Every $\frac{1}{n}$-derivation of ${\mathfrak L}$ leaves each ${\mathfrak L}_i$ invariant and its restriction to ${\mathfrak L}_i$ is trivial; hence it has the form $\sum_{i=1}^s\kappa_i\pi_i$, where $\pi_i$ is the projection onto ${\mathfrak L}_i$. The compatibility condition shows that every right multiplication $R_w$ in $({\mathfrak L},\cdot)$ is a $\frac{1}{n}$-derivation of $({\mathfrak L},[\;,\ldots,\;])$, and since $\dim{\mathfrak L}_i=n+1\geq3$, the proof of Corollary~\ref{corss} gives $x\cdot y=0$ for all $x,y\in{\mathfrak L}$.
\end{proof}

\section{$\frac{1}{2}$-derivations of Witt type algebras 
and transposed Poisson algebras}

In the present section we consider the Witt algebra and one of the more interesting generalizations of the Witt algebra. 
It is proven that the Witt algebra admits many nontrivial structures of transposed Poisson algebras
and the algebra ${\mathcal W}(a,b)$ does not admit structures of transposed Poisson algebras if and only if $b\neq -1.$
All transposed Poisson algebra structures defined on ${\mathcal W}(a,-1)$ were described. 

\subsection{$\frac{1}{2}$-derivations of the Witt algebra 
and transposed Poisson algebras}
Let ${\mathcal L}$ be the algebra of Laurent polynomials with generators $\{ e_i\}_{i \in \mathbb Z}$ and the multiplication $e_ie_j=e_{i+j}.$
The algebra of derivations of ${\mathcal L}$ gives the well-known Witt algebra. 
It is the algebra ${\mathfrak L}$ with generators $\{ e_i\}_{i \in \mathbb Z}$ and the multiplication given by the following way $[e_i,e_j]=(i-j)e_{i+j}.$ 
By \cite[Theorem 5]{fil2}, 
the space of $\frac{1}{2}$-derivations of the Witt algebra 
${\mathfrak L}$ gives an associative commutative algebra.
Let us denote the associative algebra of 
$\frac{1}{2}$-derivations of the Witt algebra 
${\mathfrak L} = \mathfrak{Der} ({\mathcal L})$ by $\Delta_{\frac{1}{2}}(\mathfrak{Der} ({\mathcal L})).$
 The following theorem is a particular case of 
 \cite[Theorem 2.1]{zus10}, but our proof is useful for the future consideration.

\begin{theorem}\label{teo_witt}
Let $\varphi$ be a $\frac{1}{2}$-derivation of the Witt algebra ${\mathfrak L}.$
Then there is a set $\{ \alpha_i \}_{i \in {\mathbb Z}}$ of elements from the basic field, such that $\varphi (e_i)=\sum\limits_{j \in \mathbb{Z}} \alpha_{j} e_{i+j}.$
Every finite set $\{ \alpha_i \}_{i \in {\mathbb Z}}$ of elements from the basic field gives a $\frac{1}{2}$-derivation of ${\mathfrak L}.$ 
In particular, $\Delta_{\frac{1}{2}}( {\mathfrak Der} ( \mathcal L)) \cong \mathcal L.$
\end{theorem}

\begin{proof}
Let $\varphi$ be a $\frac{1}{2}$-derivation of ${\mathfrak L}$ and 
$\varphi(e_i)=\sum\limits_{j \in \mathbb{Z}} \alpha_{i,j} e_j,$
then 
$$2i\sum\limits_{j \in \mathbb{Z}} \alpha_{i,j} e_j=2i\varphi(e_i)=2\varphi[e_i,e_0]=[\varphi(e_i),e_0]+[e_i,\varphi(e_0)]=$$ 
$$\sum\limits_{j \in \mathbb{Z}} \alpha_{i,j}[e_j, e_0]+\sum\limits_{j \in \mathbb{Z}} \alpha_{0,j} [e_i,e_j]=
\sum\limits_{j \in \mathbb{Z}} \alpha_{i,j}je_j +\sum\limits_{j \in \mathbb{Z}} \alpha_{0,j}(i-j) e_{i+j}.$$

Comparing the coefficients at $e_j$, we have $(2i-j)(\alpha_{i,j}-\alpha_{0,j-i})=0$, i.e., $\alpha_{i,j}=\alpha_{0,j-i}$ for $j\neq 2i$.
It is easy to see that for every finite set of numbers $\{ \alpha_{j} \}_{j \in {\mathbb Z}}$
the linear mapping $\varphi$ defined by
$\varphi(e_i)=\sum\limits_{j \in \mathbb{Z}} \alpha_{j} e_{j+i}$ is a $\frac{1}{2}$-derivation of the Witt algebra $\mathfrak L.$
In particular, $\psi=\varphi-\psi_0$, where $\psi_0(e_i)=\sum\limits_{j \in \mathbb{Z}} \alpha_{0,j-i} e_j$, is a $\frac{1}{2}$-derivation and
$\psi(e_i)=\beta_ie_{2i}$, where $\beta_i=\alpha_{i,2i}-\alpha_{0,i}$ and $\beta_0=0$. For $i\neq0$ we have
$$0=4i\beta_0e_0=2\psi[e_i,e_{-i}]=[\psi(e_i),e_{-i}]+[e_i,\psi(e_{-i})]=3i\beta_ie_i+3i\beta_{-i}e_{-i},$$
hence $\beta_i=0$. Now, for every number $i$ we have
$\varphi(e_i)=\sum\limits_{j \in \mathbb{Z}} \alpha_{0,j-i} e_j.$

Now we consider three special types of $\frac{1}{2}$-derivations.
Define 
$$\varphi_{+}(e_i)=e_{i+1}, \ \varphi_{0}(e_i)=e_{i}, \ \varphi_{-}(e_i)=e_{i-1}.$$
It is easy to see that 
$$\varphi_{+}^k(e_i)= e_{i+k}, \ \varphi_+\varphi_-(e_i)=\varphi_-\varphi_+(e_i)=e_i, \ \varphi_{-}^k(e_i)=e_{i-k}.$$
It follows that the subalgebra generated by $\{ \varphi_{+},\varphi_{0}, \varphi_{-} \}$ coincides with the algebra of $\frac{1}{2}$-derivations of the Witt algebra $\mathfrak L.$
Now we have that the algebra of $\frac{1}{2}$-derivations of $\mathfrak L$ is isomorphic to $\mathcal L.$
\end{proof}

\begin{theorem}\label{teowitt}
Let $(\mathfrak{L}, \cdot, [\;,\;])$ be a transposed Poisson algebra structure defined on the Witt algebra $(\mathfrak{L}, [\;,\;])$.
Then $(\mathfrak{L}, \cdot, [\;,\;])$ is not Poisson algebra and 
 $(\mathfrak{L}, \cdot)$  is a mutation of the algebra of Laurent polynomials.
On the other hand, every mutation of the algebra of Laurent polynomials gives a transposed Poisson algebra structure with the Lie part isomorphic to the Witt algebra.

\end{theorem}

\begin{proof}
We aim to describe the multiplication $\cdot.$
By Lemma \ref{glavlem}, 
for every element $e_n$ there is a related $\frac{1}{2}$-derivation $\varphi_n$ of $(\mathfrak{L}, [\;,\;]),$
such that $\varphi_j(e_i)=e_i \cdot e_j = \varphi_i(e_j),$
where $\varphi_i(e_j)= \sum \limits_{k \in {\mathbb Z}} \alpha_{i,k} e_{j+k}.$
Hence, 
$$ \sum\limits_{k \in {\mathbb Z}} \alpha_{j,k-i} e_{k}= \varphi_j(e_i)=e_i \cdot e_j = \varphi_i(e_j)= \sum\limits_{k \in {\mathbb Z}} \alpha_{i,k-j} e_{k} $$
and $\alpha_{i,j}=\alpha_{0,j-i}.$
It follows that 
\begin{equation}\label{product}
 e_i \cdot e_j = \sum\limits_{k \in {\mathbb Z}} \alpha_{k} e_{i+j+k}= e_ie_j w,
\end{equation}
where $w=\sum\limits_{k \in {\mathbb Z}} \alpha_{k} e_k.$
Now, analyzing the associative law of $\cdot,$ we have the following:
$$(e_i \cdot e_j) \cdot e_m= \sum\limits_{k \in {\mathbb Z}} \alpha_{k} e_{i+j+k} \cdot e_m=
 \sum\limits_{k \in {\mathbb Z}} \alpha_{k} \left( \sum\limits_{k^* \in {\mathbb Z}} \alpha_{k^*} e_{i+j+m+k+k^*}\right)=e_i \cdot (e_j \cdot e_m),$$
which implies that the multiplication $\cdot$ defined by (\ref{product}) is associative.
The algebra $({\mathfrak L}, \cdot)$ is a mutation of $\mathcal L.$
On the other hand, it is easy to see that for any element $w$ the product (\ref{product}) gives a transposed Poisson algebra structure defined on the Witt algebra $(\mathfrak{L}, [\;,\;])$.
It is easy to see that $  \mathfrak{L} \cdot [  \mathfrak{L},\mathfrak{L}]\neq 0$
and, by \cite[Proposition 2.4]{bai20},  $(\mathfrak{L}, \cdot, [\;,\;])$ is non-Poisson.
It gives the complete statement of the theorem. 
\end{proof}

\subsection{$\frac{1}{2}$-derivations of the Lie algebra $\mathcal{W}(a,b)$ and transposed Poisson algebras}
In \cite{13,16}, a class of representations $I(a, b) = \bigoplus\limits_{m \in {\mathbb Z}} {\mathbb C} I_m$ for the Witt algebra ${\bf W}$ with two
complex parameters $a$ and $b$ had been introduced. 
The action of ${\bf W}$ on $I(a, b)$ is given by
$L_m \cdot I_n = -(n + a + bm)I_{m+n}.$
$I(a, b)$ is the so called tensor density module. The Lie algebra $\mathcal{W}(a,b)={\bf W} \ltimes I(a, b),$ where ${\bf W}$ is the Witt algebra
and $I(a, b)$ is the tensor density module of ${\bf W}.$
Many maps, such that commuting maps, biderivations, $2$-local derivations, and also commutative post-Lie algebra structures
of the algebra ${\mathcal W}(a,b)$ were studied in \cite{han16, tang18, tang20}.

\begin{definition}The Lie algebra $\mathcal{W}(a,b)$ is spanned by generators 
$\{L_i, \ I_j \}_{ i,j \in \mathbb{Z}}$. 
These generators satisfy 
\[ [L_m, L_n] = (m -n)L_{m+n}, \ [L_m, I_n] = -(n + a + b m)I_{m+n}. \] 
\end{definition}

\begin{theorem}
There are no non-trivial $\frac{1}{2}$-derivations of the Lie algebra $\mathcal{W}(a,b)$ for $b\neq -1.$
Let $\varphi$ be a $\frac{1}{2}$-derivation of the algebra $\mathcal{W}(a,-1),$ then 
there are two finite sets of elements from the basic field $\{ \alpha_t \}_{t \in {\mathbb Z}}$ and $\{ \beta_t \}_{t \in {\mathbb Z}},$
such that 
$\varphi(L_m) = \sum\limits_{t \in {\mathbb Z}} \alpha_t L_{m+t} + \sum\limits_{t \in {\mathbb Z}} \beta_t I_{m+t}$ and $\varphi (I_m)= \sum\limits_{t \in {\mathbb Z}} \alpha_t I_{m+t}.$

\end{theorem}

\begin{proof}
Let $\varphi$ be a $\frac{1}{2}$-derivation of $\mathcal{W}(a,b).$
As $\mathcal{W}(a,b)$ is a $\mathbb{Z}_2$-graded algebra, we can consider $\varphi$ as a sum of two $\frac{1}{2}$-derivations $\varphi_0$ and $\varphi_1,$
such that
\[
\varphi_0( \langle L_i \rangle_{i \in {\mathbb Z}} ) \subseteq \langle L_i \rangle_{i \in {\mathbb Z}}, \ 
\varphi_0( \langle I_i \rangle_{i \in {\mathbb Z}} ) \subseteq \langle I_i \rangle_{i \in {\mathbb Z}}, \ 
\varphi_1( \langle L_i \rangle_{i \in {\mathbb Z}} ) \subseteq \langle I_i \rangle_{i \in {\mathbb Z}}, \ 
\varphi_1( \langle I_i \rangle_{i \in {\mathbb Z}} ) \subseteq \langle L_i \rangle_{i \in {\mathbb Z}}.\]

Let us consider $\varphi_0.$ Since $\langle I_i \rangle_{i \in {\mathbb Z}}$ is an ideal and the quotient algebra is the Witt algebra, by Theorem \ref{teo_witt}, it is easy to see that
$\varphi_0(L_i) = \sum\limits_{k \in {\mathbb Z}} \alpha_k L_{i+k}$ and suppose that
$\varphi_0(I_0) = \sum\limits_{k \in {\mathbb Z}} \beta_k I_{k}.$
It follows that for all $m,$ such that $a+mb \neq 0,$ we have 

\begin{longtable}{lllll}
$\varphi_0(I_m)$&$=$&$ -\frac{1}{a+mb}\varphi_0 [L_m, I_0]=
-\frac{1}{2(a+mb)} \left( \sum\limits_{k \in {\mathbb Z}} \alpha_k [L_{m+k}, I_0] + \sum\limits_{k \in {\mathbb Z}} \beta_k [ L_m, I_k] \right),$\\
$\varphi_0(I_m)$&$=$& 
$\sum\limits_{k \in {\mathbb Z}} \left( \frac{\alpha_k (a+bm+bk)+\beta_k(k+a+mb)}{2(a+mb)}\right)I_{k+m}.$
\end{longtable}
Then

\begin{longtable}{l} 
$-(-m+a+mb)\varphi(I_0)= \varphi[L_m, I_{-m}]=
\frac{1}{2} \left( [\varphi(L_m), I_{-m}]+ [L_m, \varphi(I_{-m})] \right)=$\\

$-\frac{1}{2}
\sum\limits_{k \in {\mathbb Z}} \left( \alpha_k (-m+a+(k+m)b) + \frac{(k-m+a+mb)(\alpha_k(a-bm+bk)+\beta_k(k+a-mb))}{2(a-mb)} \right)I_k.$
\end{longtable}
Hence, 

\begin{equation} \label{wabeq}
\beta_k = \alpha_k \frac{ 3 a^2 + a (k + 3 b k - 3 m) +b (k^2 - (2 + b) k m + 3 (1 - b) m^2)}
{ 3 a^2 - k^2 + k m + 3 (1 - b) b m^2 - a (2 k + 3 m)}.
\end{equation} 

The relation (\ref{wabeq}) does not determine $\alpha_k$ and $\beta_k$ for some special values of the parameters (for example, for $a=b=0$ it gives only $\beta_k=0$ for $k\neq0$). For this reason, we use all relations of the form $2\varphi_0[L_m,I_n]=[\varphi_0(L_m),I_n]+[L_m,\varphi_0(I_n)]$.
Let $\varphi_0(I_n)=\sum\limits_{k\in{\mathbb Z}}\gamma_{n,k}I_{n+k}$ (so, $\gamma_{0,k}=\beta_k$). Comparing the coefficients at $I_{m+n+k},$ we have
\begin{equation}\label{wabE}
2(n+a+bm)\gamma_{m+n,k}=\alpha_k(n+a+bm+bk)+(n+k+a+bm)\gamma_{n,k}.
\end{equation}
For $m=0$ it gives $(n+a-k)\gamma_{n,k}=\alpha_k(n+a+bk)$. Hence, $\gamma_{n,k}=\alpha_k\frac{n+a+bk}{n+a-k}$ for $n\neq k-a$, and substituting it into (\ref{wabE}) (for $n\neq k-a$ and $m+n\neq k-a$), we obtain
\begin{equation}\label{wabE1}
k\,m\,(b+1)\,(ab-2a-bk-bm+bn-2n)\,\alpha_k=0.
\end{equation}
Analyzing (\ref{wabE}) and (\ref{wabE1}), we have two cases:

\begin{enumerate}

\item $b\neq -1.$ If $k\neq 0,$ then the polynomial $ab-2a-bk-bm+bn-2n$ in $m,n$ is non-zero (the coefficients at $m$ and $n$ are $-b$ and $b-2$), and choosing suitable $m\neq0$ and $n$ in (\ref{wabE1}), we have $\alpha_k=0.$ Then $\gamma_{n,k}=0$ for $n\neq k-a$, and if $n_0=k-a\in{\mathbb Z}$, then (\ref{wabE}) with $n=n_0-m$, $m\neq0$ and $n+a+bm\neq0$ gives $\gamma_{n_0,k}=0.$
If $k=0,$ then $(n+a)(\gamma_{n,0}-\alpha_0)=0$, i.e., $\gamma_{n,0}=\alpha_0$ for $n\neq -a$; for $n_0=-a\in{\mathbb Z}$, the relation (\ref{wabE}) with $n=n_0-m$, $m\neq0$ and $n+a+bm\neq0$ gives $\gamma_{n_0,0}=\alpha_0.$
It follows that $\varphi_0=\alpha_0\,{\rm id},$
which gives that $\varphi_0$ is a trivial $\frac{1}{2}$-derivation.

Let us now consider the ``odd" part of $\varphi,$ i.e. $\varphi_1.$
It is known that the commutator of a $\frac{1}{2}$-derivation and a derivation gives a new $\frac{1}{2}$-derivation.
Then, $[\varphi_1, {\mathbb L}_{I_m}],$ where ${\mathbb L}_{I_m}$ is the left multiplication on $I_m,$ gives a $\frac{1}{2}$-derivation, such that it keeps invariant of subspaces $\langle L_i \rangle_{i \in {\mathbb Z}}$ and $\langle I_i \rangle_{i \in {\mathbb Z}}$. Thus it is a trivial $\frac{1}{2}$-derivation. Then, 
if $\varphi_1(I_k)=\sum\limits_{t \in {\mathbb Z}} \Gamma^k_t L_{t}$ and $[\varphi_1, {\mathbb L}_{I_m}]=\alpha_m\,{\rm id},$ we have
\[ \alpha_m I_k = [\varphi_1, {\mathbb L}_{I_m}](I_k)= -[I_m, \varphi_1(I_k)]=
-\sum\limits_{t \in {\mathbb Z}} \Gamma^k_t (m+a+tb)I_{m+t}, \]
which gives $\alpha_m=0$ and $\varphi_1(I_k)=0.$
Hence, if $\varphi_1(L_n)=\sum\limits_{t \in {\mathbb Z}} \gamma^n_t I_{n+t},$ then

\begin{longtable}{ll}
$2 n \varphi_1(L_n)=$& $ 2 \varphi_1[L_n,L_0] = [\varphi_1(L_n),L_0] + [L_n, \varphi_1(L_0)]=$ \\
&$\sum\limits_{t \in {\mathbb Z}} \left(\gamma^n_{t} (n+t+a) - \gamma_t^0(t+a+nb) \right)I_{n+t}$
\end{longtable}
and, for $n\neq t+a,$
\begin{longtable}{c}
$ \gamma_{t}^n = \gamma_t^0 \frac{t+a+nb}{t+a-n}.$
\end{longtable}
It follows

\begin{longtable}{ll}
$2(m-n) \varphi_1(L_{m+n})=$ &$[\varphi_1(L_m), L_n]+ [L_m, \varphi_1(L_n)] =$\\
& $\sum\limits_{t \in {\mathbb Z}} \left( (m+t+a+nb) \gamma^m_{t} - (n+t+a+mb)\gamma^n_{t} \right) I_{m+n+t}$\\
\end{longtable}
and 
\begin{longtable}{c} 
$\frac{2(m-n)(t+a +(n+m)b)}{t+a-(n+m)} \gamma^0_t=\left( \frac{(m+t+a+nb)(t+a+mb)}{t+a-m}-\frac{(n+t+a+mb)(t+a+nb)}{t+a-n}\right) \gamma^0_t,$
\end{longtable}
which gives 
\begin{longtable}{c} 
$ \frac{n m (1 + b) (m - n) (a (2 - b) + b (m + n - t) + 2 t)}{(a - m + t) (a - n + t) (a - m - n + t)} \gamma^0_t=0.$
\end{longtable}
Since $b\neq-1$, the last factor $a(2-b)+b(m+n-t)+2t$ is a non-zero polynomial in $m,n$, unless $b=0$ and $t=-a.$ In the first case, choosing suitable $m$ and $n$ we have $\gamma^0_t=0.$ In the second case, $\gamma^n_t=\gamma^0_t\frac{t+a}{t+a-n}=0$ for $n\neq0,$ and the relation $2(m-n)\varphi_1(L_{m+n})=[\varphi_1(L_m), L_n]+ [L_m, \varphi_1(L_n)]$ with $n=-m\neq0$ gives $4m\gamma^0_t=0.$
Hence, $\gamma^0_t=0$ and $\gamma^n_t=0$ for $n\neq t+a$; for $n_0=t+a\in{\mathbb Z}$, the last relation with $m+n=n_0$, $m\neq n$ and $m,n\neq n_0$ gives $\gamma^{n_0}_t=0.$
It follows that $\varphi_1=0$ and $\varphi$ is trivial.

 \item $b=-1.$ Then $(n+a-k)(\gamma_{n,k}-\alpha_k)=0$, i.e., $\gamma_{n,k}=\alpha_k$ for $n\neq k-a$; for $n_0=k-a\in{\mathbb Z}$, the relation (\ref{wabE}) with $n=n_0-m$ and $n+a-m\neq0$ gives $2(n+a-m)\gamma_{n_0,k}=2(n+a-m)\alpha_k$, i.e., $\gamma_{n_0,k}=\alpha_k.$
It gives
\begin{equation}\label{eq666}
\varphi_0(L_m)= \sum\limits_{k \in {\mathbb Z}} \alpha_k L_{m+k}\mbox{ and }\varphi_0(I_m)= \sum\limits_{k \in {\mathbb Z}} \alpha_k I_{m+k}.
\end{equation}
It is easy to see, that for every finite set $\{ \alpha_k \}_{k \in {\mathbb Z}}$ of elements from the basic field, the linear map defined as (\ref{eq666})
gives a $\frac{1}{2}$-derivation of $\mathcal{W}(a,-1).$

Let us now consider $\varphi_1$ of $\mathcal{W}(a,-1).$
Let $\varphi_1(I_n)= \sum\limits_{t \in {\mathbb Z}} \delta^n_t L_{n+t}.$ Since $[\varphi_1(L_0),I_n]=0$, the relation
$2\varphi_1[L_0, I_n]=[\varphi_1(L_0),I_n]+[L_0,\varphi_1(I_n)]$ gives $-2(n+a)\delta^n_t=-(n+t)\delta^n_t$, i.e., $\delta^n_t=0$ for $t\neq n+2a.$
Hence, $\varphi_1(I_n)=\delta_nL_{2n+2a}$ (and $\varphi_1(I_n)=0$ if $2a\notin{\mathbb Z}$), and
\begin{longtable}{l}
$0= 2 \varphi_1[I_n, I_m] = [ \varphi_1(I_n), I_m]+ [I_n,\varphi_1(I_m)]$\\
$\phantom{0}=-\delta_n(m-2n-a)I_{2n+m+2a}+\delta_m(n-2m-a)I_{2m+n+2a}.$
\end{longtable}
For $m\neq n$ the elements $I_{2n+m+2a}$ and $I_{2m+n+2a}$ are different, hence $\delta_n(m-2n-a)=0$ for all $m\neq n$, and $\delta_n=0.$
It follows that $\varphi_1(I_m)=0$ for all $m\in {\mathbb Z}.$
Now, let us consider the action of $\varphi_1$ on $\langle L_i \rangle_{i \in {\mathbb Z}}:$
let $\varphi_1(L_n)= \sum\limits_{t \in {\mathbb Z}} \gamma^n_t I_{n+t}$ and $\beta_t=\gamma^0_t$. As in the case $b\neq-1$, the relation $2n\varphi_1(L_n)=[\varphi_1(L_n),L_0] + [L_n, \varphi_1(L_0)]$ gives $(n-t-a)(\gamma^n_t-\beta_t)=0$, and the relation $2(m-n) \varphi_1(L_{m+n})=[\varphi_1(L_m), L_n]+ [L_m, \varphi_1(L_n)]$ gives $\gamma^{t+a}_t=\beta_t$ in the separate case $n=t+a.$ Hence,
\begin{equation}\label{equa2}
\varphi_1 (L_m)=
 \sum\limits_{t \in {\mathbb Z}} \beta_t I_{m+t}.
\end{equation}

It is easy to see that for any finite set $\{ \beta_k \}_{k \in {\mathbb Z}}$ the linear map defined by (\ref{equa2}) gives a $\frac{1}{2}$-derivation of $\mathcal{W}(a,-1).$
Hence, we have the statement of the theorem.
\end{enumerate}

\end{proof}

Hence, by Theorem \ref{princth}, we have the following corollary.

\begin{corollary}
There are no non-trivial transposed Poisson algebra structures defined on a Lie algebra $\mathcal{W}(a,b),$ for $b\neq -1.$
\end{corollary}

Similar to the present generalization of the Witt algebra, we can define a generalization of the algebra of Laurent polynomials in the following way: it is a direct sum of the algebra of Laurent polynomials and the regular module over it.

\begin{definition}
The algebra of extended Laurent polynomials $\mathbb L$ is a commutative algebra generated by $\langle L_i, I_j \rangle_{i, j \in {\mathbb Z}}$ and it satisfies
$L_iL_j=L_{i+j},$ $L_iI_j=I_{i+j}$ and $I_iI_j=0.$  
\end{definition}

\begin{theorem}
Let $(\mathfrak{L}, \cdot, [\;,\;])$ be a transposed Poisson algebra structure defined on the 
Lie algebra $\mathcal{W}(a,-1).$
Then $(\mathfrak{L}, \cdot, [\;,\;])$ is not Poisson algebra and 
$(\mathfrak{L}, \cdot)$ is a mutation of the algebra of extended Laurent polynomials.
On the other hand, every mutation of the algebra of extended Laurent polynomials gives a transposed Poisson algebra structure with Lie part isomorphic to $\mathcal{W}(a,-1).$
\end{theorem}

\begin{proof}
Using ideas from the proof of Theorem \ref{teowitt}, it is easy to see that there are two sets
$\{\alpha_t\}_{t \in {\mathbb Z}}$ and $\{\beta_t \}_{t \in {\mathbb Z}},$ such that 
$L_i \cdot L_j= \sum\limits_{t \in {\mathbb Z}} \alpha _t L_{i+j+t} +\sum\limits_{t \in {\mathbb Z}} \beta_t I_{i+j+t}.$
It follows that $L_i \cdot I_j =\sum\limits_{t \in {\mathbb Z}} \alpha _t I_{i+j+t}$
and for every $\frac{1}{2}$-derivation $\phi_m$ of $\mathcal{W}(a,-1)$ corresponding to the right multiplication on $I_m,$ we have that $\phi_m( \langle L_t \rangle_{t \in {\mathbb Z}}) \subseteq \langle I_t \rangle_{t \in {\mathbb Z}},$ which gives $I_i \cdot I_j=0.$
Let us define $\sum\limits_{t \in {\mathbb Z}} \alpha _t L_{t} +\sum\limits_{t \in {\mathbb Z}} \beta_t I_{t}$ as $w.$
Then 
\begin{equation}\label{prowab}
L_i \cdot L_j = L_iwL_j, \ L_i \cdot I_j = L_iwI_j, \ I_i \cdot I_j = I_iwI_j=0, 
\end{equation}
which gives that $(\mathfrak{L}, \cdot)$ is a mutation of the extended algebra of Laurent polynomials.
On the other hand, it is easy to see that for any element $w$ product (\ref{prowab}) gives a transposed Poisson algebra structure defined on the Lie algebra $\mathcal{W}(a,-1)$ for an arbitrary complex number $a.$
It is easy to see that $  \mathfrak{L} \cdot [  \mathfrak{L},\mathfrak{L}]\neq 0$
and, by \cite[Proposition 2.4]{bai20},  $(\mathfrak{L}, \cdot, [\;,\;])$ is non-Poisson.
It gives the complete statement of the theorem. 
\end{proof}

\section{$\frac{1}{2}$-derivations of Virasoro type
algebras and transposed Poisson algebras}
The present section is devoted to the study of
$\frac{1}{2}$-derivations of the Virasoro algebra and some Lie superalgebras related to the Virasoro algebra.
Namely, 
we describe all $\frac{1}{2}$-superderivations of $N=1$ and $N=2$ conformal superlgebras.
As a corollary, we obtain that there are no non-trivial transposed Poisson structures defined  by
the Virasoro algebra, in both $N=1$ and $N=2$ conformal superalgebras.

\subsection{$\frac{1}{2}$-derivations of Virasoro algebra
and transposed Poisson algebras}
The Virasoro algebra is the unique central extension of the Witt algebra, considered in the previous section.

\begin{definition}
The Virasoro algebra $\mathbf{Vir}$ is spanned by generators $\{ L_n \}_{ n \in \mathbb{Z}}$ and the central charge $c$. These generators satisfy 
 $$
 [L_m,L_n] = (m-n)L_{m+n} + \dfrac{m^3 - m}{12}\delta_{m+n,0}c.
 $$
\end{definition}

\begin{theorem}\label{virasoro}
There are no non-trivial $\frac{1}{2}$-derivations of the Virasoro algebra $\mathbf{Vir}.$
\end{theorem}

\begin{proof}
Let $\varphi$ be a $\frac{1}{2}$-derivation of $\mathbf{Vir}$
 and $\varphi(L_n) = \sum\limits_{i \in \mathbb{Z}} \alpha_{n,i} L_i + \rho_n c.$ We then get
\begin{longtable}{lll}
 $[\varphi(L_p),L_q]$&$ =$&$ \sum\limits_{i \in \mathbb{Z}} (i-q) \alpha_{p,i}L_{i+q} + \dfrac{i^3-i}{12}\delta_{i+q,0}\alpha_{p,i} c$ \\
 $[L_p,\varphi(L_q)]$&$ =$&$ \sum\limits_{j \in \mathbb{Z}} (p-j) \alpha_{q,j}L_{p+j} + \dfrac{p^3-p}{12}\delta_{p+j,0}\alpha_{q,j}c$
\end{longtable}
and
\begin{longtable}{lll}
 $\varphi[L_p,L_q]$&$ = $&$(p-q)\sum\limits_{k \in \mathbb{Z}}\alpha_{p+q,k} L_{k} +(p-q)\rho_{p+q}c + \dfrac{p^3 - p}{12} \delta_{p+q,0}\varphi(c).$ 
\end{longtable}
  It is easy to see that \begin{center}
      $0=2\varphi[L_p,c]=[\varphi(L_p),c]+[L_p,\varphi(c)]=[L_p,\varphi(c)].$
  \end{center}
  Hence,  $\varphi(c) = \rho_c c.$ Since the identity map is obviously a $\frac{1}{2}$-derivation we then can put $\varphi(c)=0$, i.e., $\rho_c = 0.$

Now for every numbers $i,j$ we have 
\[
 [\varphi(L_p),L_q] = \sum_{i\in \mathbb{Z}} (i - 2q) \alpha_{p,i-q}L_i + \dfrac{q-q^3}{12} \alpha_{p,-q} c,
\]
and
\[
 [L_p,\varphi(L_q)] = \sum_{j \in \mathbb{Z}} (2p - j) \alpha_{q,j-p}L_j + \dfrac{p^3 - p}{12}\alpha_{q,-p}c.
\]

Adding up similar terms we obtain
\begin{multline*}
[\varphi(L_p),L_q] + [L_p,\varphi(L_q)] = \sum_{k\in \mathbb{Z}} ((k-2q)\alpha_{p,k-q} + (2p-k)\alpha_{q,k-p})L_k \\
 + \dfrac{c}{12} \left( (q-q^3) \alpha_{p,-q} + (p^3 - p)\alpha_{q,-p}\right).\end{multline*}

Thus for any $k\in \mathbb{Z}$ we get
\begin{equation}\label{eq1}
 (p-q)\alpha_{p+q,k} = \dfrac{1}{2} \left( (k-2q)\alpha_{p,k-q} + (2p-k)\alpha_{q,k-p}\right) 
\end{equation}
and
\begin{equation}\label{eq2}
 (p-q)\rho_{p+q} = \dfrac{1}{24}( (q-q^3) \alpha_{p,-q} + (p^3 - p)\alpha_{q,-p}) 
\end{equation}
 
Set $q=0$, by (\ref{eq1}),
\[
 p\alpha_{p,k} = \dfrac{1}{2}( k\alpha_{p,k} + (2p-k)\alpha_{0,k-p}) 
\]
hence $(2p-k)(\alpha_{p,k}-\alpha_{0,k-p})=0$. The case $k=2p$ is not covered by this relation, but since $\varphi(c)=0$, the map $\varphi$ induces a $\frac{1}{2}$-derivation of the quotient algebra $\mathbf{Vir}/\langle c\rangle$, which is the Witt algebra, and by Theorem \ref{teo_witt} we have
\begin{equation}\label{eq3}
 \alpha_{p,k} = \alpha_{0,k-p} \mbox{ for all } p,k\in{\mathbb Z}.
\end{equation}

By (\ref{eq2},\ref{eq3}), for $p\neq q$,
\[
 (p-q)\rho_{p+q} = \dfrac{1}{24}( (q-q^3) \alpha_{0,-q-p} + (p^3 - p)\alpha_{0,-p-q})= 
 \dfrac{p-q}{24}( p^2+pq+q^2-1)\alpha_{0,-p-q}
\]
and
\[
\rho_{p+q} = \dfrac{1}{24}( p^2+pq+q^2-1)\alpha_{0,-p-q},\]
which gives that $\rho_{p}=0$ and $\alpha_{0,p}=0$ for all numbers $p.$
Hence, $\varphi=0$
and every $\frac{1}{2}$-derivation of $\mathbf{Vir}$ can be expressed as multiplication on some fixed element from the basic field. 
 \end{proof}

Hence, by Theorem \ref{princth}, we have the following corollary. 

\begin{corollary}

There are no non-trivial transposed Poisson algebra structures defined on the Virasoro algebra $\mathbf{Vir}$.
\end{corollary}

\subsection{$\frac{1}{2}$-derivations of super Virasoro algebra
and transposed Poisson superalgebras}
In mathematical physics, a super Virasoro algebra is an extension of the Virasoro algebra to a Lie superalgebra. There are two extensions with particular importance in superstring theory: the Ramond algebra (named after Pierre Ramond) and the Neveu–Schwarz algebra (named after André Neveu and John Henry Schwarz). Both algebras have $N = 1$ supersymmetry and an even part given by the Virasoro algebra. They describe the symmetries of a superstring in two different sectors, called the Ramond sector and the Neveu–Schwarz sector.

\begin{definition}
Let $\mathbf{sVir}$ be a Lie superalgebra generated by 
even elements $\{ L_m,\ c \}_{m\in \mathbb{Z}}$ and 
odd elements $\{ G_{r} \}_{ r \in i+\mathbb{Z}} $, 
where $i=0$ (the Ramond case), or $i=\frac{1}{2}$ (the Neveu–Schwarz case). 
In both cases, $c$ is central in the superalgebra, and the additional graded brackets are given by
\[ [L_{m},G_{r}]=\left({\frac {m}{2}}-r\right)G_{m+r} \ \mbox{ and }
\ [G_{r},G_{s}]=2L_{r+s}+{\frac {c}{3}}\left(r^{2}-{\frac {1}{4}}\right)\delta _{r+s,0}. \]

\end{definition}

\begin{lemma}\label{supervir}
There are no non-trivial $\frac{1}{2}$-derivations of the super Virasoro algebra $\mathbf{sVir}.$
\end{lemma}

\begin{proof}

Let $\varphi$ be a $\frac{1}{2}$-superderivation of $\mathbf{sVir}.$
Then the even part $\varphi_0$ and the odd part $\varphi_1$ of $\varphi$ are, respectively, 
an even $\frac{1}{2}$-superderivation and an odd $\frac{1}{2}$-superderivation of $\mathbf{sVir}.$
The restriction $\varphi_0|_{\mathbf{Vir}}$ is a $\frac{1}{2}$-derivation of $\mathbf{Vir}$ and it is trivial by Theorem \ref{virasoro}. Hence, we can suppose that $\varphi_0|_{\mathbf{Vir}}=0.$
Let $r_0\in\{0,\pm\frac12\}\cap(i+{\mathbb Z})$. Since $[L_m,G_{r_0}]=(\frac{m}{2}-r_0)G_{m+r_0}$, we have
\[\Big(\frac{m}{2}-r_0\Big)\varphi_0(G_{m+r_0})=\frac{1}{2}[L_m, \varphi_0(G_{r_0})]. \]
Moreover, $[G_{r_0},G_{r_0}]\in\mathbf{Vir}$, hence $0=\varphi_0[G_{r_0},G_{r_0}]= [G_{r_0}, \varphi_0(G_{r_0})]$. If $\varphi_0(G_{r_0})=\sum_s\gamma_sG_s$, then
$0=\sum_s\gamma_s\big(2L_{r_0+s}+\frac{c}{3}(r_0^2-\frac14)\delta_{r_0+s,0}\big)$, which gives $\varphi_0(G_{r_0})=0$, and $\varphi_0(G_{m+r_0})=0$ for $m\neq 2r_0$.
Taking $r_0=0$ in the Ramond case and $r_0=\pm\frac12$ in the Neveu--Schwarz case, we have $\varphi_0=0.$

It is known that
the supercommutator of a $\frac{1}{2}$-superderivation and a superderivation gives a new $\frac{1}{2}$-superderivation.
Hence, $[\varphi_1, {\rm ad}_{G_i}]=\varphi_1{\rm ad}_{G_i}+{\rm ad}_{G_i}\varphi_1$ is an even $\frac{1}{2}$-superderivation of $\mathbf{sVir},$ and it is trivial: $[\varphi_1, {\rm ad}_{G_i}]=\alpha_i\,{\rm id}$.
If $\varphi_1(c) =\sum\limits_{k} \gamma_k G_k,$ then
\[
\alpha_i c= [\varphi_1, {\rm ad}_{G_i}](c)= [G_i, \varphi_1(c)]=
2 \sum\limits_{k} \gamma_k L_{i+k} + \frac{c}{3}\Big(i^2-\frac{1}{4}\Big) \gamma_{-i},
\]
which gives $\varphi_1(c)=0$ and $\alpha_i=0.$ Hence, $\varphi_1[G_i,x]=-[G_i,\varphi_1(x)]$ for all $x\in\mathbf{sVir}.$ Comparing this with the definition of a $\frac{1}{2}$-superderivation, we have
\[\varphi_1[L_k,G_i]=[\varphi_1(L_k),G_i]=[L_k,\varphi_1(G_i)] \ \mbox{ and } \ \varphi_1[G_i,G_j]=[\varphi_1(G_i),G_j].\]
Let $\varphi_1(G_i)=\sum\limits_{t\in {\mathbb Z}} \Gamma^i_t L_t+\Gamma^ic,$ then
\[
    -\sum\limits_{t\in {\mathbb Z}} t \Gamma^i_tL_t=[L_0, \varphi_1(G_i)]=\varphi_1[L_0, G_i]=-i\varphi_1(G_i)=-i\sum\limits_{t\in {\mathbb Z}} \Gamma^i_tL_t-i \Gamma^i c,
\]
which gives $\varphi_1(G_i)=0$ in the Neveu--Schwarz case, and $\varphi_1(G_i)=\Gamma^i_i L_i$ for $i\neq 0$ in the Ramond case. In the Ramond case, for $i\neq0$ and $m+i\neq0$, the relation
$\varphi_1[L_m, G_i]=[L_m,\varphi_1(G_i)]$ gives $(\frac{m}{2}-i)\Gamma^{m+i}_{m+i}=(m-i)\Gamma^i_i$, i.e., $(j-3i)\Gamma^j_j=2(j-2i)\Gamma^i_i$ for all $i,j\neq0$.
Taking $(i,j)=(1,2)$, $(2,j)$ with $j\neq 6$, and $(1,6)$, we have $\Gamma^j_j=0$ for all $j\neq0$. Now, $\frac{m}{2}\varphi_1(G_m)=[L_m,\varphi_1(G_0)]$ gives $\Gamma^0_0=0$, and $\frac{3m}{2}\varphi_1(G_0)=[L_m,\varphi_1(G_{-m})]=0$ ($m\neq0$) gives $\Gamma^0=0.$
Thus, $\varphi_1(\mathbf{sVir}_1)=0$ in both cases, and 
$2\varphi_1(L_{i+j})=\varphi_1[G_i,G_j]=[\varphi_1(G_i),G_j]=0$ for $i+j\neq0$, and $2\varphi_1(L_0)=\varphi_1[G_i,G_{-i}]=0$. It follows that $\varphi_1=0$ and $\varphi$ is trivial.

\end{proof}

Hence, by Lemma \ref{superteo}, we have the following corollary. 

\begin{corollary}

There are no non-trivial transposed Poisson superalgebra structures defined on a super Virasoro algebra.
\end{corollary}

\subsection{$\frac{1}{2}$-derivations of $N = 2$ superconformal algebra
and transposed Poisson superalgebras}
In mathematical physics, the $2D$ $N = 2$ superconformal algebra is an infinite-dimensional Lie superalgebra, related to supersymmetry, that occurs in string theory and two-dimensional conformal field theory. It has important applications in mirror symmetry. It was introduced by Ademollo, Brink, and D'Adda et al. in 1976 as a gauge algebra of the $U(1)$ fermionic string. 
There are two slightly different ways to describe the $N = 2$ superconformal algebra, called the $N = 2$ Ramond algebra and the $N = 2$ Neveu–Schwarz algebra, which are isomorphic (see below) but differ in the choice of standard basis. 
 
\begin{definition}
The $N = 2$ superconformal algebra is the Lie superalgebra with basis of even elements $\{c, L_n, J_n\}_{n \in {\mathbb Z}},$ and odd elements $\{ G^+_r,$ $G^-_r\}_{r \in i+{\mathbb Z}},$ $i=0$ (the Ramond basis), or $i=\frac{1}{2}$ (the Neveu–Schwarz basis). 
Additional graded brackets are given by
\begin{longtable}{ll}
$[L_{m},L_{n}]=(m-n)L_{{m+n}}+{c \over 12}(m^{3}-m)\delta _{{m+n,0}}$& 
$[L_{m},\,J_{n}]=-nJ_{m+n}$\\
$[J_{m},J_{n}]={c \over 3}m\delta _{m+n,0}$ &
$ [J_{m},G_{r}^{\pm }]=\pm G_{m+r}^{\pm } $\\

$[G_{r}^{+},G_{s}^{-}]=L_{r+s}+{1 \over 2}(r-s)J_{r+s}+{c \over 6}(r^{2}-{1 \over 4})\delta _{r+s,0}$ &
$ [L_{m},G_{r}^{\pm }]=({m \over 2}-r)G_{r+m}^{\pm } $
\end{longtable}
\end{definition} 

\begin{lemma}
There are no non-trivial $\frac{1}{2}$-derivations of $N = 2$ superconformal algebra.
\end{lemma}

\begin{proof}
The main idea of the proof is similar to the proof of Lemma \ref{supervir} and we will give only a brief sketch to prove it.
Let $\mathcal{N}=\mathcal{N}_0+\mathcal{N}_1$ be a $N = 2$ superconformal algebra and $\varphi=\varphi_0+\varphi_1$ be a $\frac{1}{2}$-superderivation of $\mathcal{N}.$
Now, 
$\mathcal{N}_0=(\mathcal{N}_0)_0+(\mathcal{N}_0)_1$ is a $\mathbb{Z}_2$-graded algebra, where
$(\mathcal{N}_0)_0$ is generated by $\{ L_n, c \}_{n \in {\mathbb Z}}.$ It is isomorphic to the Virasoro algebra, $(\mathcal{N}_0)_1$ is generated by $\{ J_n\};$
and $(\varphi_0)|_{\mathcal{N}_0}=
((\varphi_0)|_{\mathcal{N}_0})_0+((\varphi_0)|_{\mathcal{N}_0})_1
$ is a $\frac{1}{2}$-derivation of $\mathcal{N}_0.$
By Theorem \ref{virasoro},
$((\varphi_0)|_{\mathcal{N}_0})_0$ is a trivial $\frac{1}{2}$-derivation of $(\mathcal{N}_0)_0$
and we can consider $((\varphi_0)|_{\mathcal{N}_0})_0=0.$
By some easy calculations similar with the proof of Lemma \ref{supervir}, 
we have $(\varphi_0)|_{\mathcal{N}_0}=0$ and it is a trivial $\frac{1}{2}$-derivation of $\mathcal{N}_0.$
After that, we have that every even $\frac{1}{2}$-superderivaton $\varphi_0$ of $\mathcal{N}_0$ is trivial.
Now, let $\phi_x$ be an inner odd derivation (left multiplication on $x \in \mathcal{N}_1$) of $\mathcal{N},$ then 
$[\varphi_1, \phi_x]$ is an even $\frac{1}{2}$-derivation of $\mathcal{N},$ which is trivial.
Hence, it is easy to see that $\varphi_1=0$ and $\varphi$ is trivial.
\end{proof}

Hence, by Lemma \ref{superteo}, we have the following corollary. 

\begin{corollary}
There are no non-trivial transposed Poisson superalgebra structures defined on a $N = 2$ superconformal algebra.
\end{corollary}

\section{$\frac{1}{2}$-derivations of thin Lie algebra 
and transposed Poisson algebras}

There are several papers that deal with the problem of classification of nilpotent algebras (for example, low dimensional algebras, 
filiform algebras, algebras with an abelian ideal, etc). 
Nilpotent Lie algebras with an abelian ideal of codimension $1$ were described in \cite{khakha}.
Local derivations of nilpotent Lie algebras with an abelian ideal of codimension $1$ were studied in \cite{tang20}.
In the present section, we describe all transposed Poisson structures on nilpotent Lie algebras with an abelian ideal of codimension $1.$
In our proof we consider the infinite-dimensional nilpotent Lie algebra with an abelian ideal of codimension $1,$
but the statement of the theorem can be adapted for the finite-dimensional case.

\begin{definition}
The thin Lie algebra $\mathfrak{L}$ is spanned by generators $\{e_i\}_{i \in \mathbb{N} }.$ 
These generators satisfy 
 \[
 [e_1,e_n] = e_{n+1}, \ n >1.
 \]
 We denote the operator of left multiplication on $e_1$ by $L.$
\end{definition}

\begin{lemma}\label{lenthin}
Let $\varphi$ be a $\frac{1}{2}$-derivation of $\mathfrak{L}.$
Then 
\begin{longtable}{rclrclrcl}
$\varphi (e_1)$& $=$& $\sum\limits_{i \in \mathbb{N}} \alpha_i e_i,$ \\
$\varphi (e_2)$&$ =$& $\sum\limits_{i \geq 2} \beta_i e_i,$ \\
$\varphi (e_{n})$&$ = $ & $(1-2^{2-n}) \alpha_1 e_n+ 2^{2-n} L^{n-2} \varphi(e_2).$
\end{longtable}
\end{lemma}

\begin{proof}
Let $\varphi (e_1)=\sum\limits_{i \in \mathbb{N}} \alpha_i e_i$ and $\varphi (e_2)=\sum\limits_{i \in \mathbb{N}} \beta_i e_i,$ then
\begin{longtable}{lll}
$\varphi(e_n)$&$=$&$\frac{1}{2}\left([\varphi(e_1),e_{n-1}]+[e_1, \varphi(e_{n-1})]\right)=
\frac{1}{2} \alpha_1 e_{n} + \frac{1}{2}L \varphi(e_{n-1}), \quad n\geq3.$
\end{longtable}
Hence, by induction,  the formula for $\varphi(e_n)$ follows. In particular, $\varphi(e_3)$ has zero coefficient at $e_1$, and
$0=2\varphi[e_2,e_3]=[\varphi(e_2),e_3]+[e_2,\varphi(e_3)]=\beta_1e_4$, i.e., $\beta_1=0.$  

\end{proof}

\begin{theorem}
Let $(\mathfrak{L}, \cdot, [\;,\;])$ be a transposed Poisson algebra structure defined on the thin Lie algebra $(\mathfrak{L}, [\;,\;])$.
Then $(\mathfrak{L}, \cdot, [\;,\;])$ is not Poisson algebra and there are $u\in\mathbb{C}$ and $X\in\langle e_i\rangle_{i\geq2}$, $(u,X)\neq(0,0)$, such that the nonzero products of $\cdot$ are
$$e_1\cdot e_1=ue_1+X, \quad e_1\cdot e_n=ue_n, \ n\geq2;$$
conversely, every such multiplication gives a transposed Poisson algebra structure on $(\mathfrak{L}, [\;,\;])$. Up to isomorphism, $(\mathfrak{L}, \cdot, [\;,\;])$ is one of the following structures:
\begin{enumerate}
 \item $(\mathfrak{L}, *, [\;,\;]),$ where $e_1*e_1=e_1,$ $e_1*e_n=e_n,$ $n\geq2;$
 \item $(\mathfrak{L}, *, [\;,\;]),$ where $e_1*e_1=X=\sum_{j\geq2}\xi_je_j\neq0.$
\end{enumerate}
The structures (2) defined by $X$ and $X'$ are isomorphic if and only if $X'=\mu\sum_{j\geq2}\lambda^{j-4}\xi_je_j$ for some $\lambda,\mu\in\mathbb{C}^*$. In particular, the structures $e_1*e_1=e_k$, $k\geq2$, are pairwise non-isomorphic, but they do not exhaust the structures (2): for example, the structures $e_1*e_1=e_2+e_3+te_4$, $t\in\mathbb{C}$, are pairwise non-isomorphic and are not isomorphic to any structure $e_1*e_1=e_k$.
\end{theorem}

\begin{proof}
We aim to describe the multiplication $\cdot.$
By lemma \ref{glavlem}, 
for every element $e_n$ there is a related $\frac{1}{2}$-derivation $\varphi_n$ of $(\mathfrak{L}, [\;,\;]),$
such that $\varphi_j(e_i)=e_i \cdot e_j = \varphi_i(e_j).$
Now, by Lemma \ref{lenthin},  
\begin{longtable}{rclrclrcl}
$\varphi_i (e_1)$& $=$& $\sum\limits_{j \in \mathbb{N}} \alpha_{ij} e_j,$ \\
$\varphi_i (e_2)$&$ =$& $\sum\limits_{j \in \mathbb{N}} \beta_{ij}e_j,$ \\
$\varphi_i (e_{n})$&$ = $ & $(1-2^{2-n}) \alpha_{i1} e_n+ 2^{2-n} L^{n-2} \varphi_i(e_2).$\\
\end{longtable}
Let us use the following notations
\[\varphi_1 (e_1)= \sum_{r=1}^p \alpha_r e_r=x, \; 
\varphi_2 (e_1)=\sum_{r=1}^q \beta_r e_r= y, \; 
\varphi_2 (e_2)=\sum_{r=1}^t \gamma_r e_r= z. \]
Then
\[ e_1 \cdot e_1 = \varphi_1 (e_1)= x, \; e_2 \cdot e_1=e_1 \cdot e_2=\varphi_1 (e_2)=\varphi_2 (e_1)= y, \; e_2 \cdot e_2= \varphi_2 (e_2)= z. \]
 
On the other hand, if $i,j,n \geq 3,$ then
\begin{longtable}{lcl} 
$e_1 \cdot e_n$& $=$ & 
$\varphi_1(e_n)= (1-2^{2-n}) \alpha_{11} e_n+ 2^{2-n} L^{n-2} \varphi_1(e_2)$\\ & $=$ & $(1-2^{2-n}) \alpha_{11} e_n+ 2^{2-n} L^{n-2} y,$\\

$e_2 \cdot e_n$& $=$ & 
$\varphi_2(e_n)= (1-2^{2-n}) \alpha_{21} e_n+ 2^{2-n} L^{n-2} \varphi_2(e_2)$\\ & $=$ & $(1- 2^{2-n})\alpha_{21} e_n+ 2^{2-n} L^{n-2} z,$\\
\end{longtable}
and
\begin{longtable}{lllllllllllll} 
$e_i \cdot e_j$ &$ = $ & $(1-2^{2-j})\alpha_{i1} e_j+ 2^{2-j} L^{j-2} \varphi_i(e_2)$\\
&$ = $ & 
$(1-2^{2-j}) \alpha_{i1} e_j+ 
 (2^{2-j}-2^{4-i-j})\alpha_{21} L^{j-2}e_i$\\ &&$\qquad+ 2^{4-i-j} L^{i+j-4} z,$\\

$e_j \cdot e_i$ &$ = $ & $(1-2^{2-i})\alpha_{j1} e_i+ 2^{2-i} L^{i-2} \varphi_j(e_2)$\\
&$ = $ & 
$(1- 2^{2-i}) \alpha_{j1} e_i+ 
 (2^{2-i}-2^{4-i-j})\alpha_{21} L^{i-2}e_j$\\ &&$\qquad+ 2^{4-i-j} L^{i+j-4} z.$\\
\end{longtable}
From the last two relations, for $i \neq j,$ we have 
\begin{longtable}{lcrl}
$\alpha_{j1}=\alpha_{21}=0,$ & \; which gives \; &$e_i \cdot e_j=2^{4-i-j} L^{i+j-4} z,$ & for all $i,j \geq 2.$
\end{longtable}

Now we have defined the commutative multiplication $\cdot,$
 but it is also an associative multiplication.
 Let us study the associative property of the multiplication $\cdot.$

For $i\neq j, n \geq 2,$ we have

\begin{longtable}{llllllllllllll}

$e_i\cdot (e_j \cdot e_n)$ &$=$& $ e_i \cdot (2^{4-j-n} L^{j+n-4} z)$
&$=$&
$2^{4-j-n}( \sum\limits_{r=1}^{t} 2^{8-i-j-n-r} \gamma_r L^{i+j+n+r-8} z ) $\\
\end{longtable}
and on the other hand
\begin{longtable}{llllllllllllll}
$e_j\cdot (e_i \cdot e_n)$ &$=$& $ e_j \cdot (2^{4-i-n} L^{i+n-4} z)$
&$=$&
$2^{4-i-n}( \sum\limits_{r=1}^{t} 2^{8-i-j-n-r} \gamma_r L^{i+j+n+r-8} z ) $\\
\end{longtable}
 it implies that  $z=0.$

Summarizing, 
we have that there are only the following nonzero products:
\begin{longtable}{lllllllll} 
$e_1 \cdot e_1$ & $=$ & $x,$ &$ e_1 \cdot e_n$ & $=$ & 
$(1- 2^{2-n}) \alpha_{1} e_n+ 2^{2-n} L^{n-2} y.$
\end{longtable}

Now, we are interested in the associative relation for $e_1, e_1$ and $e_n:$
\begin{longtable}{lllllllll} 
$e_1 \cdot (e_1 \cdot e_n)$ 
&$=$&$e_1 \cdot ((1- 2^{2-n})\alpha_{1} e_n+ 2^{2-n} L^{n-2} y)$\\
&$=$&$(1- 2^{2-n}) \alpha_{1} e_1 \cdot e_n+ 2^{2-n} \sum\limits_{r=1}^q \beta _re_1 \cdot e_{n-2+r}$
\end{longtable}
and on the other hand, we have
\begin{longtable}{lllllllll} 
$e_n\cdot (e_1 \cdot e_1)$ &$=$&
$e_n\cdot x $ &$=$&
$ \alpha_1 e_1 \cdot e_n,$
\end{longtable}
which gives
\begin{equation}\label{link}
2^{2-n}\Big(\sum\limits_{r=2}^q\beta_r\, e_1\cdot e_{n+r-2}-\alpha_1\, e_1\cdot e_n\Big)=0
\end{equation}
(recall that $\beta_1=0$ by Lemma \ref{lenthin}, since $y=\varphi_1(e_2)$).
If $q>2$, then the highest term of $\beta_q\,e_1\cdot e_{n+q-2}$ is $2^{4-n-q}\beta_q^2 e_{n+2q-4}$, while all other terms in (\ref{link}) have smaller indexes; hence $\beta_q=0$, a contradiction. So $y=\beta_2e_2$ and $e_1\cdot e_n=\big((1-2^{2-n})\alpha_1+2^{2-n}\beta_2\big)e_n$ for $n\geq2$.
Now (\ref{link}) gives $(\beta_2-\alpha_1)\big((1-2^{2-n})\alpha_1+2^{2-n}\beta_2\big)=0$ for all $n\geq2$. If $\beta_2\neq\alpha_1$, then $n=2$ gives $\beta_2=0$ and $n=3$ gives $\alpha_1=0$, a contradiction. Hence, $\alpha_1=\beta_2=:u$ and
$$e_1\cdot e_1=ue_1+X, \quad e_1\cdot e_n=ue_n, \ n\geq2, \quad \mbox{where } X=\sum\limits_{r\geq2}\alpha_re_r.$$
Conversely, for all $u$ and $X$ this multiplication is associative and commutative, and its right multiplications $R_{e_1}=u\,{\rm id}+\psi_X$ (where $\psi_X(e_1)=X$, $\psi_X(e_n)=0$, $n\geq2$) and $R_{e_m}$ ($R_{e_m}(e_1)=ue_m$, $R_{e_m}(e_n)=0$), $m\geq2$, are $\frac{1}{2}$-derivations of $(\mathfrak{L}, [\;,\;])$; hence it gives a transposed Poisson algebra structure.
It is not a Poisson algebra: if $u\neq0$, then $[e_2,e_1\cdot e_1]=-ue_3\neq-2ue_3=[e_2,e_1]\cdot e_1+e_1\cdot[e_2,e_1]$, and if $u=0$, then $[e_1,e_1\cdot e_1]=[e_1,X]\neq0=2[e_1,e_1]\cdot e_1$.

Let us describe these structures up to isomorphism. Every automorphism $\theta$ of $(\mathfrak{L}, [\;,\;])$ has the form
$$\theta(e_1)=\lambda e_1+r, \quad \theta(e_n)=\mu\lambda^{n-2}e_n, \ n\geq2, \quad \mbox{where } \lambda,\mu\in\mathbb{C}^*, \ r\in\langle e_i\rangle_{i\geq2}.$$
Indeed, $\theta$ preserves the ideal $\langle e_i\rangle_{i\geq2}$ (the centralizer of $[\mathfrak{L},\mathfrak{L}]$), hence $\theta(e_1)=\lambda e_1+r$ and $\theta(e_n)=\lambda^{n-2}L^{n-2}\theta(e_2)$; if $\theta(e_2)\notin\mathbb{C}e_2$, then $e_2$ does not belong to the image of $\theta$.
If $u\neq0$, then in the basis $E_1=u^{-1}e_1-u^{-2}X$, $E_n=u^{2-n}e_n$, $n\geq2$, we have $[E_1,E_n]=E_{n+1}$, $E_1\cdot E_1=E_1$ and $E_1\cdot E_n=E_n$, which gives the structure (1).
If $u=0$, then $X\neq0$, and an automorphism $\theta$ is an isomorphism between the structures defined by $X=\sum\xi_je_j$ and $X'$ if and only if $\theta(X)=\theta(e_1)\cdot'\theta(e_1)=\lambda^2X'$, i.e., $X'=\mu\sum_{j\geq2}\lambda^{j-4}\xi_je_j$.
Finally, if $\theta$ is an isomorphism between structures with parameters $u$ and $u'$, then $u=\lambda u'$; hence the structures (1) and (2) are not isomorphic, and the statement follows.

\end{proof}

\section{$\frac{1}{2}$-derivations of the solvable Lie algebra with abelian nilpotent radical 
and transposed Poisson algebras}

Several papers deal with the problem of classification of all solvable Lie algebras with a given nilradical (for example, Abelian, Heisenberg, filiform, quasi-filiform nilradicals, etc. \cite{wint}).
Local derivations of solvable Lie algebras with an abelian radical of codimension $1$ were studied in \cite{ayupov}.
In the present section, we describe all transposed Poisson structures on solvable Lie algebras with abelian radical of codimension $1.$
In our proof we consider the solvable infinite-dimensional   Lie algebra with an abelian radical of codimension $1$ 
but the statement of the theorem can be adapted for the finite-dimensional case.

\begin{definition}
Let $\mathfrak{L}$ be the solvable Lie algebra with abelian nilpotent radical of codimension $1$ which spanned by generators $\{e_i\}_{i \in \mathbb{N}}$ satisfying
\[ [e_1, e_{n}]= e_{n}, \ n\geq 2. \]
\end{definition}

\begin{lemma}\label{lemab1}
Let $\varphi$ be a $\frac{1}{2}$-derivation of $\mathfrak{L}.$
Then 
\begin{longtable}{rclrclrcl}
$\varphi (e_1)$& $=$& $ \alpha e_1+ \sum\limits_{i \geq 2} \alpha_i e_i,$ &
$\varphi (e_n)$&$ =$& $ \alpha e_n,$ & 
\end{longtable}
\end{lemma}

\begin{proof}
Let $\varphi (e_n)= \sum\limits_{i \in {\mathbb N}} \alpha_{ni} e_i,$ then if $n \geq 2,$ we have
\begin{longtable}{l}
$\varphi(e_n)=\frac{1}{2}\left([\varphi(e_1),e_{n}]+[e_1, \varphi(e_{n})]\right)=
\frac{1}{2} \alpha_{11} e_{n} + \frac{1}{2} \varphi(e_{n}) - \frac{1}{2} \alpha_{n1} e_1.$
\end{longtable}
Hence, $\varphi(e_n)=\alpha_{11}e_n-\alpha_{n1}e_1$, which gives $\alpha_{n1}=0$ and $\varphi(e_n)=\alpha_{11}e_n$, and we have the statement of the Lemma.

\end{proof}

\begin{theorem}
Let $(\mathfrak{L}, \cdot, [\;,\;])$ be a transposed Poisson algebra structure defined on the solvable Lie algebra with abelian nilpotent radical of codimension $1$, $(\mathfrak{L}, [\;,\;])$.
Then $(\mathfrak{L}, \cdot, [\;,\;])$ is not Poisson algebra and there are $\alpha\in\mathbb{C}$ and $X\in\langle e_i\rangle_{i\geq2}$, $(\alpha,X)\neq(0,0)$, such that the nonzero products of $\cdot$ are $e_1\cdot e_1=\alpha e_1+X$, $e_1\cdot e_n=\alpha e_n$, $n\geq2$; conversely, every such multiplication gives a transposed Poisson algebra structure. Up to isomorphism, $(\mathfrak{L}, \cdot, [\;,\;])$ is one of the following structures
\begin{enumerate}
 \item $(\mathfrak{L}, *, [\;,\;]),$ where $e_1*e_1 =e_2;$ 
 \item $(\mathfrak{L}, *, [\;,\;]),$ where $e_1*e_n =\alpha e_n, \ n \geq 1,$ $\alpha\in\mathbb{C}^*.$ 
\end{enumerate}
All these structures are pairwise non-isomorphic.

\end{theorem}

\begin{proof}
We aim to describe the multiplication $\cdot.$
By Lemma \ref{glavlem}, 
for every element $e_n$ there is a  corresponding   $\frac{1}{2}$-derivation $\varphi_n$ of $(\mathfrak{L}, [\;,\;]),$
such that $\varphi_j(e_i)=e_i \cdot e_j = \varphi_i(e_j)$
and $\varphi_1(e_1) = \alpha e_1 +\sum\limits_{i \geq 2} \alpha_i e_i.$
  By Lemma \ref{lemab1},  
\begin{enumerate}
 \item if $i,j \geq 2,$ then $e_i \cdot e_j =0;$
 \item if $i \geq 2,$ then $e_1 \cdot e_i = \varphi_1 (e_i)= \alpha e_i;$
 \item $e_1\cdot e_1= \varphi_1(e_1)= \alpha e_1 +\sum\limits_{i \geq 2} \alpha_ie_i.$
\end{enumerate}
It is easy to see that $\cdot$ is an associative commutative multiplication.

  Conversely, for all $\alpha$ and $X=\sum_{i\geq2}\alpha_ie_i$ the right multiplications of this algebra are $\frac{1}{2}$-derivations of $(\mathfrak{L}, [\;,\;])$, and we obtain a transposed Poisson algebra structure.
Note that every automorphism $\theta$ of the Lie algebra $(\mathfrak{L}, [\;,\;])$ satisfies $\theta(e_1)=e_1+r$, $r\in\langle e_i\rangle_{i\geq2}$, and $\theta(\langle e_i\rangle_{i\geq2})=\langle e_i\rangle_{i\geq2}$ (the coefficient of $e_1$ in $\theta(e_1)$ equals $1$, since $[\theta(e_1),\theta(e_n)]=\theta(e_n)$); conversely, every such linear bijection is an automorphism.
\begin{enumerate}
 \item if $\alpha = 0,$ then $X\neq0$; let $X=\sum\limits_{i \geq u} \alpha_i e_i$ with $\alpha_u \neq 0.$ 
 By choosing 
 \[E_1= e_1, \ E_2=X, \
 E_k=e_{k-1}, \ 3\leq k \leq u,\ E_k=e_{k}, \ u+1\leq k,\]
 we have an isomorphic algebra structure with the following table of multiplications:
 \[E_1*E_1= E_2.\]
 
 \item if $\alpha \neq 0,$ then 
 by choosing \[E_1= e_1-\alpha^{-1}X, \ E_k= e_{k}, \ k \geq 2,\]
 we have an isomorphic algebra structure with the following table of multiplications:
 \[ E_1*E_n =\alpha E_n, \ n \geq 1.\] 
\end{enumerate}
If $\theta$ is an isomorphism between structures with parameters $\alpha$ and $\alpha'$, then $\alpha\theta(e_2)=\theta(e_1\cdot e_2)=(e_1+r)\cdot'\theta(e_2)=\alpha'\theta(e_2)$, i.e., $\alpha=\alpha'$. Hence, the structures (1) and (2) are pairwise non-isomorphic.
It is not a Poisson algebra: if $\alpha=0$, then $[e_1,e_1\cdot e_1]=[e_1,X]=X\neq0=2[e_1,e_1]\cdot e_1$, and if $\alpha\neq0$, then $[e_2,e_1\cdot e_1]=-\alpha e_2\neq-2\alpha e_2=[e_2,e_1]\cdot e_1+e_1\cdot[e_2,e_1]$.

Hence, we have the statement of the theorem.

\end{proof}

\appendix
\section*{Appendix A. Differences from the published version}
\addcontentsline{toc}{section}{Appendix A. Differences from the published version}

\noindent\textbf{Theorems 18 and 21, Lemma 29:} proofs corrected.

\noindent\textbf{Theorems 36 and 39:} statements changed and proofs corrected.

\end{document}